\documentclass{amsart}
\usepackage{amsmath}
\usepackage{amssymb}

\newtheorem{theorem}{Theorem}[section]
\newtheorem{lemma}[theorem]{Lemma}
\newtheorem{proposition}[theorem]{Proposition}
\newtheorem{corollary}[theorem]{Corollary}
\theoremstyle{definition}
\newtheorem{definition}[theorem]{Definition}
\newtheorem{example}[theorem]{Example}

\theoremstyle{remark}

\numberwithin{equation}{section}
\newcommand{\eps}{\varepsilon}

\begin{document}
\title[Fixed point results]{Some fixed point results in ordered partial metric spaces}
\author[H. Aydi
]
{
Hassen Aydi
} \maketitle

\begin{abstract}
In this paper, we establish some fixed point theorems in ordered partial metric spaces. An example is given to illustrate
our obtained results.

\end{abstract}

\vspace{.25cm}\noindent\textbf{2000 Mathematics Subject
Classification.} 47H10, 54H25.

\vspace{.25cm}\noindent \textbf{Key Words and Phrases}:
Fixed point, partial metric space, ordered set.

\section{\textbf{INTRODUCTION AND PRELIMINARIES}}
\noindent When fixed point problems in partially ordered metric spaces are concerned, first results were obtained by Ran and
Reurings \cite{RR}, and then by Nieto and L\'opez \cite{NL}. The following fixed point theorem was proved in
these papers.
 \begin{theorem}
 \label{main-thm0}\cite{NL, RR}
 Let $(X,\leq_X)$ be a partially ordered set and let d be a metric on X such that $(X,d)$ is a complete metric
space.
  Let $f: X\rightarrow X$ be a non-decreasing map with respect to $\leq_X$. Suppose that the
following conditions hold:\\
(i) $\exists 0 \leq c < 1$, \, $d(fx,fy)\leq c d(x, y)$ for any $y\leq_X
x$\\
(ii) $\exists  x_0 \in X$ such that $x_0\leq_X fx_0$\\
(iii) $f$ is continuous in $(X,d)$, or\\
(iii') if a non-decreasing sequence $\{x_n\}$ converges to $x\in
X$, then $x_n\leq_X x$ for all $n$.\\
Then $f$ has a fixed point $u\in X$.
\end{theorem}
Results on weakly contractive mappings in such spaces were obtained by Harjani and Sadarangani in \cite{HS}.
An extension of the previous result is the following
\begin{theorem}
 \label{main-thm00}\cite{HS}
Let $(X,\leq)$ be a partially ordered set and suppose that there
exists a metric $d$ in $X$ such that $(X, d)$ is a complete metric
space. Let $f: X\rightarrow X$ be a non-decreasing
mapping with respect to $\leq_X$ such that
\[
d(fx, fy)\leq d(x, y)-\psi(d(x, y)),
\]
 for $y\leq_X x$,  where $\psi: [0,+\infty[\rightarrow [0,+\infty[$ is a continuous and non-decreasing function such that it is
positive in $]0,+\infty[$, $\psi(0)=0$ and
$\displaystyle\lim_{t\rightarrow +\infty}\psi(t)=+\infty$. Assume that\\
(i) $f$ is continuous in $(X,d)$, or\\
(ii) if a non-decreasing sequence $\{x_n\}$ converges to $x\in
X$, then $x_n\leq_X x$ for all $n$.\\
If
there exists $x_0 \in X$ with $x_0 \leq_X fx_0$, then $f$ has a fixed point.
\end{theorem}
 Many other results on the existence of fixed points  or common fixed
points in ordered metric spaces were given, we can cite for example
\cite{ALD, AS, beg, BL, ciric, lak, NaSa, OP, S, V} and the references therein.

In this paper we extend the results of Harjani and Sadarangani \cite{HS} to the case of partial metric spaces. An example is considered to illustrate our obtained results.

First, we start with some preliminaries on partial
metric spaces. For more details, we refer the reader to \cite{AE, ASS, Ma2, O1, O2, OV, Ro, Ro1, Ro2, S1, V}.
\begin{definition}
\label{partial} Let $X$ be a nonempty set. A partial metric on $X$ is a
function $p : X \times X\longrightarrow R_+$  such that for all
$x, y, z \in
X$:\\
\rm{(p1)} $x =y\Longleftrightarrow p(x, x) = p(x, y) = p(y, y)$,\\
\rm{(p2)} $p(x, x) \leq  p(x,y)$,\\
\rm{(p3)} $p(x, y)= p(y,x),$\\
\rm{(p4)} $p(x,y)\leq p(x,z)+ p(z,y) - p(z,z)$.\\
 A partial metric
space is a pair $(X, p)$ such that $X$ is a nonempty set and $p$
is a partial metric on $X$.
\end{definition}
If $p$ is a partial metric on $X$, then the function $p^s : X
\times X\longrightarrow R_+$ given by
\[
p^s(x, y) = 2p(x, y)- p(x, x)-p(y, y),
\]
 is a metric on $X$.
\begin{definition}
\label{c-c}
Let $(X, p)$ be a partial metric space. Then: \\
\rm{(i)} a sequence $\{x_n\}$ in a partial metric space $(X, p)$
converges to a point $x\in X$ if and only if $p(x, x) =
\displaystyle\lim_{n\longrightarrow +\infty} p(x,
x_n)$;\\
\rm{(ii)} a sequence $\{x_n\}$ in a partial metric space $(X, p)$
converges properly to a point $x\in X$ if and only if $p(x, x) =
\displaystyle\lim_{n\longrightarrow +\infty} p(x_n,
x_n)=\displaystyle\lim_{n\longrightarrow +\infty} p(x, x_n)$,  if
and only if $\displaystyle\lim_{n\longrightarrow +\infty} p^s(x,
x_n)=0$;\\
\rm{(iii)} A sequence $\{x_n\}$ in a partial metric space $(X, p)$
is called a Cauchy sequence if there
exists (and is finite) $\displaystyle\lim_{n,m\longrightarrow +\infty} p(x_n, x_m)$; \\
\rm{(iv)} A partial metric space $(X, p)$ is said to be complete
if every Cauchy sequence $\{x_n\}$ in $X$ converges to a point
$x\in X$, that is $p(x, x)
=\displaystyle\lim_{n,m\longrightarrow +\infty} p(x_n, x_m)$.\\
\end{definition}
\begin{lemma}
\label{lem1}
 Let $(X, p)$ be a partial metric space.\\
(a) $\{x_n\}$ is a Cauchy sequence in $(X, p)$ if and only if it
is a Cauchy sequence in the metric space $(X, p^s)$;\\
(b) A partial metric space $(X, p)$ is complete if and only if the
metric space $(X, p^s)$ is complete. Furthermore,
$\displaystyle\lim_{n\longrightarrow +\infty} p^s(x_n, x) = 0$ if
and only if
\[
p(x, x) = \displaystyle\lim_{n\longrightarrow +\infty} p(x_n, x)
=\displaystyle\lim_{n,m\longrightarrow +\infty} p(x_n, x_m).
\]
\end{lemma}
\begin{definition}
\label{cont} Suppose that $(X_1,p_1)$ and $(X_2,p_2)$ are partial
metrics. Denote $\tau(p_1)$ and $\tau(p_2)$ their respective
topologies. We say $T:(X_1,p_1)\rightarrow (X_2,p_2)$ is
continuous if both $T:(X_1,\tau(p_1))\rightarrow (X_2,\tau(p_2))$
and $T:(X_1,\tau(p_1^s))\rightarrow (X_2,\tau(p^s_2))$ are
continuous.
\end{definition}
\begin{proposition}
\label{continuity} Let $(X,p)$ be a partial metric space,
partially ordered and $T: X\rightarrow X$ be a given mapping. We
say that $T$ is continuous in $x_0\in X$ if for every sequence
$\{x_n\}$ is $X$, we have\\
(a) $x_n$ converges to $x_0$ in $(X,p)$ implies $Tx_n$ converges to $Tx_0$ in $(X,p)$.\\
(b) $x_n$ converges properly to $x_0$ in $(X,p)$ implies $Tx_n$ converges properly to $Tx_0$ in $(X,p)$.\\
If $T$ is continuous on each point $x_0\in X$, then we say that
$T$ is continuous on $X$.
\end{proposition}
\begin{definition}
\label{increasing} If $(X,\leq_X)$ is a partially ordered set and
$f: X\rightarrow X$, we say that f is monotone nondecreasing if
$x, y \in X$, $x\leq_X y$ implies $fx \leq_X fy$.
\end{definition}

\section{Main results}
Our first result is the following theorem
\begin{theorem}
\label{main-thm2} Let $(X,\leq_X)$ be a partially ordered set and
let $p$ be a partial metric on $X$ such that $(X, p)$ is complete.
  Let $f: X\rightarrow X$ be a non-decreasing map with respect to $\leq_X$. Suppose that the
following conditions hold: for $y \leq x$, we have\\
(i)
 \begin{equation}
 \label{contraction1}
p(fx, fy)\leq p(x, y)-\psi(p(x, y)),
\end{equation}
where $\psi: [0,+\infty[\rightarrow [0,+\infty[$ is a continuous and non-decreasing function such that it is
positive in $]0,+\infty[$, $\psi(0)=0$ and
$\displaystyle\lim_{t\rightarrow +\infty}\psi(t)=+\infty$;\\
 (ii)
$\exists  x_0 \in X$ such that $x_0\leq_X fx_0$;\\
(iii) $f$ is continuous in $(X,p)$, or;\\
(iii') if a non-decreasing sequence $\{x_n\}$ converges to $x\in
X$, then $x_n\leq _X x$ for all $n$.\\
Then $f$ has a fixed point $u\in X$. Moreover, $p(u,u)=0$.
\end{theorem}

\begin{proof}
Let $x_0\in X$ be such that $x_0\leq_X fx_0$. As
$f$ is monotone non-decreasing, then
\[
x_0\leq_X fx_0\leq_X f^2x_0\leq_X f^3 x_0\leq_X...\leq_X f^n x_0\leq_X f^{n+1} x_0\leq_X...
\]
Put $x_n=f^n x_0$, then for any $n\in \mathbb{N}^*$, we have $x_{n-1}\leq_X  x_n$.
Then for each integer $n\geq 1$, from (\ref{contraction1}) and, as the elements $x_n$ and $x_{n-1}$ are comparable, we get
\begin{equation}
 \label{eq6}
p(x_{n+1},x_n)=p(fx_n, fx_{n-1})\leq p(x_n, x_{n-1})-\psi(p(x_n, x_{n-1})),
\end{equation}
If there exists $n_1 \in \mathbb{N}^*$ such that $p(x_{n_1}, x_{n_1-1})= 0$, then $x_{n_1-1}=x_{n_1}=
fx_{n_1-1}$ and $x_{n_1-1}$ is a fixed point of $f$ and the proof is
finished. In other case, suppose that $p(x_{n+1},x_n)\neq  0$ for all $n\in \mathbb{N}$. Then, using an assumption on $\psi$ in (\ref{eq6}) we have for $n\in \mathbb{N}^*$
\[
p(x_{n+1},x_n)\leq p(x_n, x_{n-1})-\psi(p(x_n, x_{n-1}))<p(x_n, x_{n-1}).
\]
Put $\rho_n=p(x_{n+1},x_n)$, then we have
\begin{equation}
 \label{eq7}
\rho_n\leq \rho_{n-1}-\psi(\rho_{n-1})<\rho_{n-1}.
\end{equation}
Therefore $\{\rho_n\}$ is a nonnegative non-increasing sequence and hence possesses a limit $\rho^*$. From (\ref{eq7}), taking limit when $n\rightarrow +\infty$,
we get
\[
\rho^*\leq \rho^*-\psi(\rho^*),
\]
and, consequently,  $\psi(\rho^*)= 0$. By our assumptions on $\psi$, we conclude $\rho^*= 0$, that is,
\begin{equation}
\label{77}
\lim_{n\rightarrow+\infty} p(x_n,x_{n+1})=0.
\end{equation}
In what follows we shall show that $\{x_n\}$ is a Cauchy sequence in the partial metric space $(X,p)$. Fix $\eps>0$, as $\rho_n=p(x_{n+1},x_n)\rightarrow 0$, there exists $n_0\in \mathbb{N}$ such that
\begin{equation}
 \label{eq8}
p(x_{n_0+1},x_{n_0})\leq \min\{\frac{\eps}{2},\psi(\frac{\eps}{2})\}.
\end{equation}
We claim that if $z\in X$ verifies $p(z,x_{n_0})\leq \eps$ and $x_{n_0}\leq_X z$, we get $p(fz,x_{n_0})\leq \eps$. Indeed, to do this we distinguish two cases :\\
Case 1. $p(z,x_{n_0})\leq \frac{\eps}{2}$.
In this case, as $z$ and $x_{n_0}$ are comparable, we have
\[
\begin{split}
p(fz,x_{n_0})\leq& p(fz,fx_{n_0})+p(fx_{n_0},x_{n_0})\\
= &p(fz,fx_{n_0})+p(x_{n_0+1},x_{n_0})\\
\leq& p(z,x_{n_0})-\psi(p(z,x_{n_0}))+p(x_{n_0+1},x_{n_0})\\
\leq&  p(z,x_{n_0})+p(x_{n_0+1},x_{n_0})\\
\leq& \frac{\eps}{2}+\frac{\eps}{2}=\eps.
\end{split}
\]
Case 2. $\frac{\eps}{2}\leq p(z,x_{n_0})\leq \eps$.
In this case, as  $\psi$ is a non-decreasing function, $\psi(\frac{\eps}{2})\leq \psi(p(z,x_{n_0}))$. Therefore,  from (\ref{eq8}) we have
\[
\begin{split}
p(fz,x_{n_0})\leq& p(fz,fx_{n_0})+p(fx_{n_0},x_{n_0})\\
= &p(fz,fx_{n_0})+p(x_{n_0+1},x_{n_0})\\
\leq& p(z,x_{n_0})-\psi(p(z,x_{n_0}))+p(x_{n_0+1},x_{n_0})\\
\le& p(z,x_{n_0})-\psi(\frac{\eps}{2})+p(x_{n_0+1},x_{n_0})\\
\leq& p(z,x_{n_0})-\psi(\frac{\eps}{2})+\psi( \frac{\eps}{2})\\
\leq& p(z,x_{n_0})\leq \eps.
\end{split}
\]
This proves the claim. As $x_{n_0+1}$ verifies  $p(x_{n_0+1},x_{n_0})\leq \eps$ and $x_{n_0}\leq_X x_{n_0+1}$, the claim gives us that
$x_{n_0+2}=fx_{n_0+1}$ verifies $ p(x_{n_0+2},x_{n_0})\leq \eps$. We repeat this process to get
\[
p(x_n,x_{n_0})\leq \eps\quad\mbox{for any}\quad n\geq n_{0}.
\]
It follows that
\[
p(x_n,x_{m})\leq p(x_n,x_{n_0})+p(x_{n_0},x_m)\leq\eps+\eps=2\eps\quad\mbox{for any}\quad n,m\geq n_{0}.
\]
Since $\eps$ is arbitrary, then $\displaystyle\lim_{n,m\rightarrow +\infty}p(x_n,x_m)=0$. Thus, $\{x_n\}$ is a Cauchy sequence in $(X,p)$, so by Lemma \ref{lem1}, $\{x_n\}$ is a Cauchy sequence in the metric space $(X,p^s)$. Since $(X,p)$ is complete, hence $(X,p^s)$ is complete, so there exists $u\in X$ such that
\begin{equation}
\label{eq9}
\displaystyle\lim_{n\rightarrow +\infty} p^s(x_n,u)=0.
\end{equation}
Thus, by Lemma \ref{lem1}, from the condition (p2) and (\ref{77}), we get
\begin{equation}
\label{eq10}
p(u,u)=\displaystyle\lim_{n\rightarrow +\infty} p(x_n,u)=\displaystyle\lim_{n\rightarrow +\infty} p(x_n,x_n)=0.
\end{equation}
We prove now that $fu=u$.
 We shall distinguish the cases (iii) and (iii') of the Theorem
\ref{main-thm2}.\\
Case 1. Suppose that the mapping $f$ is continuous. In particular,
thanks to condition $(b)$ in proposition \ref{continuity}, we have
$fx_{n}$ converges properly to $fu$ in $(X,p)$, that is
$p^s(fx_{n},fu)\longrightarrow 0$, since
$p^s(x_{n},u)\longrightarrow 0$, i.e, $\{x_{n}\}$ converges
properly to $u$ in $(X,p)$. Hence we have $\{fx_{n}\}$ converges to $fu$
in $(X,p^s)$. On the other hand, $\{fx_{n}=x_{n+1}\}$ converges to
$u$ in $(X,p^s)$ because of (\ref{eq9}). By uniqueness of the
limit in metric space $(X,p^s)$, we deduce that $fu=u$.\\
Case 2. Suppose now that the condition $(iii')$ of the theorem holds.
The sequence $\{x_n\}$ is non-decreasing with respect to $\leq_X$,
and it follows that $x_n\leq_X u$. Take $x=x_{n}$ and $y=u$
(which are comparable) in (\ref{contraction1}) to obtain that
\begin{equation}
\label{eq55}
p(u, fu) \leq p(u, x_{n+1}) + p(x_{n+1}, fu) \leq p(u, x_{n+1}) + p(u, x_n)-\psi(p(u, x_n)).
\end{equation}
 Letting $n\rightarrow +\infty$ in (\ref{eq55})
we find using (\ref{eq10}) and the properties of $\psi$ that
\[
p(fu,u)\leq 0-\psi(0)=0,
\]
hence $p(fu,u)=0$, so $fu=u$. This completes the proof of Theorem
\ref{main-thm2}.
\end{proof}
\begin{corollary}
 \label{cor1}
 Let $(X,\leq_X)$ be a partially ordered set and let $p$ be a partial metric on $X$ such that $(X, p)$ is complete.
  Let $f: X\rightarrow X$ be a non-decreasing map with respect to $\leq_X$. Suppose that the
following conditions hold:\\
(i) $\exists 0\leq c < 1$ such that
\begin{equation}
\label{contraction}
p(fx,fy)\leq c p(x, y)\quad\mbox{for any}\quad y\leq_X
x.
\end{equation}
(ii) $\exists  x_0 \in X$ such that $x_0\leq_X fx_0$;\\
(iii) $f$ is continuous in $(X,p)$, or;\\
(iii') if a non-decreasing sequence $\{x_n\}$ converges to $x\in
X$, then $x_n\leq_X x$ for all $n$.\\
Then $f$ has a fixed point $u\in X$. Moreover, $p(u,u)=0$.
\end{corollary}
{\bf Proof.} We take $\psi(t)=(1-c)t$ in Theorem \ref{main-thm2}. \\

Next theorem gives a sufficient condition for the uniqueness of the fixed point.
\begin{theorem}
\label{main-thm3}
 Let all the conditions of Theorem  \ref{main-thm2} be
fulfilled and let the following condition be satisfied: for arbitrary two
points $x, y \in X$ there exists $z \in X$ which is comparable
with both $x$ and $y$. Then the fixed point of $f$ is unique.
\end{theorem}
\begin{proof}
 Let $u$ and $v$ be two fixed points of
$f$, i.e., $fu=u$ and $fv=v$. Consider the following two cases:\\
Case 1. $u$ and $v$ are comparable. Then we can apply condition
(\ref{contraction1}) and obtain that
\[
p(u,v)=p(fu,fv)\leq p(u,v)-\psi(p(u,v)),
\]
hence $\psi(p(u,v))\leq 0$, i,e, $p(u,v)=0$, so $u=v$, that is
the uniqueness of the fixed point of $f$.\\
Case 2. Suppose now that $u$ and $v$ are not comparable. Choose an
element $w \in X$ comparable with both of them. Then also, $u = f^n
u$ is comparable with $f^n w$ for each $n$ (since $f$ is
non-decreasing). Applying (\ref{contraction1}), one obtains for $n\in \mathbb{N}^*$ that
\[
\begin{split}
p(u,f^n w)&=p(ff^{n-1}u,ff^{n-1}w)\leq p(f^{n-1}u,f^{n-1}w)-\psi(p(f^{n-1}u,f^{n-1}w))\\
& \leq p(f^{n-1}u,f^{n-1}w)= p(u,f^{n-1}w).
\end{split}
\]
 It follows
that the sequence $\{p(u, f^nw)\}$ is non-increasing and it has a
limit $l\geq 0$. Assuming that $l > 0$ and passing to the limit in
the relation
\[
p(u,f^n w)\leq p(u,f^{n-1}w)-\psi(p(u,f^{n-1}w)),
\]
one obtains that $l=0$, a contradiction. In the same way it can be
deduced that $p(v,f^n w) \rightarrow 0$ as $n\rightarrow +\infty$.
Now, passing to the limit in $p(u, v)\leq  p(u, f^n w)+ p(f^n w,
v)$, it follows that $p(u,v)=0$, so $u=v$, and the uniqueness of
the fixed point is proved.
\end{proof}
\begin{example}
Let $X=[0,+\infty[$ endowed with the usual partial metric $p$
defined by $p:X\times X\rightarrow [0,+\infty[$ with
$p(x,y)=\max\{x,y\}$. We give the partial order on $X$ by
\[
x\leq_X y\Longleftrightarrow p(x,x)=p(x,y)\Longleftrightarrow
x=\max\{x,y\}\Longleftarrow y\leq x.
\]
 It is clear  that $(X,\leq_X)$ is totally ordered. The partial metric space
 $(X,p)$ is complete because $(X,p^s)$ is complete. Indeed,
 for any $x,y\in X$,
 \[
 \begin{split}
 p^s(x,y)=2p(x,y)-p(x,x)-p(y,y)=&2\max\{x,y\}-(x+y)\\
 =& |x-y|,
 \end{split}
 \]
 Thus, $(X,p^s)=([0,+\infty[,|.|)$  is the usual metric space, which is complete.
Again, we define
\[
f(t)=\frac{1}{2}t,\quad \mbox{if}\quad t\geq 0.
\]
The function $f$ is continuous on $(X,p)$. Indeed, let $\{x_n\}$ be a sequence
 converging to $x$ in $(X,p)$, then
\[
\lim_{n\rightarrow +\infty}\max\{x_n,x\}=\lim_{n\rightarrow +\infty} p(x_n,x)=p(x,x)=x
\]
hence by definition of $f$, we have
\begin{equation}
\label{eq32}
\lim_{n\rightarrow +\infty} p(fx_n,fx)=\lim_{n\rightarrow +\infty}\max\{fx_n,fx\}=\lim_{n\rightarrow +\infty}\frac{1}{2}\max\{x_n,x\}=\frac{1}{2}x=p(fx,fx),
\end{equation}
that is $\{f(x_n)\}$ converges to $f(x)$ in $(X,p)$. On the other
hand, if $\{x_n\}$ converges properly to $x$ in $X$, hence
\[
\lim_{n\rightarrow +\infty} p^s(x_n,x)=0.
\]
Thus, by definition of $p^s$ and $f$, one can find
\begin{equation}
\label{eq33}
\lim_{n\rightarrow +\infty} p^s(fx_n,fx)=0.
\end{equation}
Both convergences (\ref{eq32})-(\ref{eq33}) yield that $f$ is
continuous on $(X,p)$. Any $x,y\in X$ are comparable, so
for example we take $x\leq_X y$,  and then $p(x,x)=p(x,y)$, so
$y\leq x$. Since $f(y)\leq f(x)$, so $f(x)\leq_X f(y)$, giving that $f$ is monotone non-decreasing with respect to $\leq_X$. In particular, for any $x\leq_X y$, we have
\begin{equation}
\label{eq13}
p(x,y)=x,\quad p(fx,fy)=f(x)=\frac{1}{2}x.
\end{equation}
Let us take $\psi : [0,+\infty[\rightarrow [0,+\infty[$ such that $\psi(t)=\frac{1}{4}t$. We have  for any $x\in X$, $\frac{1}{2}x\leq x-\frac{1}{4}x$. Consequently, we get for any $x\leq_X y$, thanks to this and (\ref{eq13})
\[
p(fx,fy)\leq p(x,y)-\psi(p(x,y),
\]
that is (\ref{contraction1}) holds. All the hypotheses of Theorem \ref{main-thm2} are satisfied, so $f$
has a unique fixed point in $X$, which is  $u=0$.\\
\end{example}

{\bf Acknowledgements:} The author thanks the referees for their kind comments and suggestions to improve this paper.

\vspace{0.2cm}

\noindent Hassen Aydi:\newline Universit\'e de Monastir.\newline
Institut Sup\'erieur d'Informatique de Mahdia. Route de R\'ejiche,
Km 4, BP 35, Mahdia 5121, Tunisie.\newline Email-address:
hassen.aydi@isima.rnu.tn\newline


\begin{thebibliography}{99}

\bibitem{ALD}
I. Altun and G. Durmaz, \textit{Some fixed point theorems on ordered cone metric spaces}, Rend. Circ. Mat. Palermo. 58 (2009), 319–325.

\bibitem{AE}
I. Altun and A. Erduran, \textit{Fixed point theorems for monontone mappings on partial metric spaces}, Fixed Point Theory Appl. 2011 (2011), Article ID 508730.

\bibitem{AS}
I. Altun and H. Simsek, \textit{Some fixed point theorems on ordered metric spaces and application}, Fixed Point Theory Appl. 2010 (2010), Article ID 621469.


\bibitem{ASS}
I. Altun, F. Sola and H. Simsek, \textit{Generalized contractions on partial metric spaces}, Topology Appl. 157 (18) (2010), 2278-2785.


\bibitem{beg}
I. Beg and A. R. Butt, \textit{Coupled fixed points of set valued mappings in partially ordered metric spaces}, J. Nonlinear Sci. Appl. 3 (3) (2010), 179-185.


\bibitem{BL}
T. Gnana Bhaskar and V. Lakshmikantham, \textit{Fixed point theorems in partially ordered metric spaces and applications}, Nonlinear Anal. 65 (2006), 1379-1393.

\bibitem{ciric}
Lj. \'Ciri\'c, N. Caki\'c, M. Rajovi\'c and J. Sheok Ume, \textit{Monotone generalized nonlinear contractions in partially ordered metric spaces}, Fixed Point Theory Appl. 2008 (2008), Article ID 131294.

\bibitem{HS}
J. Harjani and K. Sadarangani, \textit{Fixed point theorems for weakly contractive mappings in partially ordered sets}, Nonlinear Anal. 71 (2009), 3403-3410.

\bibitem{lak}
V. Lakshmikantham and Lj. \'Ciri\'c, \textit{Coupled fixed point theorems for nonlinear contractions in partially ordered metric spaces},
 Nonlinear Anal. 70 (2009), 4341-4349.

\bibitem{Ma2}
S. G. Matthews, \textit{Partial metric topology, in: Proc. 8th Summer Conference on General Topology and Applications}, in: Ann. New York Acad. Sci. 728 (1994), 183-197.


\bibitem{NL}
J. J. Nieto and  R. R. L\'opez, \textit{Contractive mapping theorems in partially ordered sets and applications to ordinary differential equations}, Order. 22 (2005), 223-239.

\bibitem{NaSa}
H. K. Nashine and B. Samet, \textit{Fixed point results for mappings satisfying $(\psi,\varphi)$-weakly contraction in partially ordered metri spaces}, Nonlinear Analysis (2010), doi: 10.1016/j.na.2010.11.024.


\bibitem{O1}
S. J. O'Neill, \textit{Two topologies are better than one}, Tech. report, University of Warwick, Coventry, UK, http://www.dcs.warwick.ac.uk/reports/283.html, 1995.

\bibitem{O2}
S. J. O'Neill, \textit{Partial metrics, valuations and domain theory}, in: Proc. 11th Summer Conference on General Topology and Applications, in: Ann. New York Acad. Sci. 806 (1996) 304-315.

\bibitem{OP}
D. O'Regan and A. Petrusel, \textit{Fixed point theorems for generalized contractions in ordered metric spaces}, J. Math. Anal. Appl. 341 (2008), 1241-1252.

\bibitem{OV}
S. Oltra and O. Valero, \textit{Banach's fixed point theorem for partial metric spaces}, Rend. Istit. Mat. Univ. Trieste. 36 (2004), 17-26.

\bibitem{RR}
A. C. M. Ran and M. C. B. Reurings, \textit{A fixed point theorem in partially ordered sets
and some applications to matrix equations}, Proc. Amer. Math. Soc. 132 (5) (2003), 1435-1443.

\bibitem{Ro}
S. Romaguera, \textit{A Kirk type characterization of completeness for partial metric spaces}, Fixed Point Theory Appl. 2010 (2010), Article ID 493298.

\bibitem{Ro1}
S. Romaguera and M. Schellekens, \textit{Partial metric monoids and semivaluation spaces},  Topology Appl. 153 (5-6) (2005), 948-962.

\bibitem{Ro2}
S. Romaguera and O. Valero, \textit{A quantitative computational model for complete partialmetric spaces via formal balls}, Math. Structures Comput. Sci. 19 (3) (2009), 541-563.

\bibitem{S}
B. Samet, \textit{Coupled fixed point theorems for a generalized Meir-Keeler contraction
in partially ordered metric spaces}, Nonlinear Anal. 72 (2010), 4508-4517.

\bibitem{S1}
M.P. Schellekens, \textit{The correspondence between partial metrics and semivaluations}, Theoret. Comput. Sci. 315 (2004), 135-149.

\bibitem{V}
O. Valero, \textit{On Banach fixed point theorems for partial metric spaces}, Appl. Gen. Topol. 6 (2) (2005), 229-240.

\end{thebibliography}
\end{document}